\documentclass[11pt,reqno]{amsart}
\usepackage{amsmath}
\usepackage{amsthm}
\usepackage{amsfonts}
\usepackage{array}
\usepackage{booktabs}
\usepackage{tabularx}
\usepackage{mathabx}
\usepackage{graphicx}
\usepackage{amssymb}
\usepackage[english]{babel}
\usepackage[colorlinks = true,
            linkcolor = blue,
            urlcolor  = blue,
            citecolor = blue,
            anchorcolor = blue]{hyperref}

\usepackage{cleveref}
\usepackage{calc}
\usepackage{dsfont}
\usepackage{enumitem}
\usepackage{placeins}

\usepackage{graphicx}
\usepackage{subcaption}

\usepackage{float}

\usepackage{physics}
\usepackage{tikz}

\allowdisplaybreaks
\usepackage[a4paper, lmargin=0.15\paperwidth, rmargin=0.15\paperwidth,
            tmargin=0.1\paperheight, bmargin=0.1\paperheight]{geometry}
\usepackage[all]{nowidow}
\usepackage{lipsum}

\theoremstyle{plain}
\newtheorem{theorem}{Theorem}[section]

\newtheorem{lemma}[theorem]{Lemma}

\newtheorem{corollary}{Corollary}[section]

\theoremstyle{definition}

\theoremstyle{remark}

\newtheoremstyle{named}{}{}{\itshape}{}{\bfseries}{.}{.5em}{\thmnote{#3}#1}
\theoremstyle{named}

\setlist[enumerate]{itemsep=0pt, topsep=2pt}

\usepackage{algorithm}
\usepackage{algpseudocode}

\usepackage{etoolbox}

\makeatletter
\patchcmd{\@setauthors}
  {\MakeUppercase{\authors}}
  {\authors}
  {}{}
\makeatother

\newcommand{\Z}{\mathbb Z}
\newcommand{\card}[1]{\lvert #1\rvert}

\newcommand{\one}{\mathbf 1}

\title[Settling the Optimal Exponent Relating Sumsets and Difference Sets]
{Settling the Optimal Exponent\\Relating Sumsets and Difference Sets}

\author[Lin and Li]{%
  \begin{tabular}{@{}c@{\qquad\qquad}c@{}}
    Haowei Lin
    &
    Shanda Li
    \\
    {\normalfont\scriptsize Tencent Hunyuan}
    &
    {\normalfont\scriptsize Carnegie Mellon University}
    \\
    {\normalfont\scriptsize
      \href{mailto:linhaowei@pku.edu.cn}{%
        \texttt{linhaowei@pku.edu.cn}%
      }%
    }
    &
    {\normalfont\scriptsize
      \href{mailto:shandal@cs.cmu.edu}{%
        \texttt{shandal@cs.cmu.edu}%
      }%
    }
  \end{tabular}%
}

\date{}

\begin{document}

\begin{abstract}
\vspace{1em}
For a finite nonempty subset \(A\) of an abelian group, let $\sigma(A)=\card{A+A}/\card A$, $\delta(A)=\card{A-A}/\card A$.
The classical sum--difference inequalities state that $$\sigma(A)^{1/2}\le\delta(A)\le\sigma(A)^2.$$
The exponent $2$ in the second inequality is known to be optimal, whereas it has remained open whether the exponent $1/2$ in the first inequality can be improved.  We settle this question by constructing an explicit family of finite
sets $A_K\subset\Z$ such that
\[
 \frac{\log\sigma(A_K)}{\log\delta(A_K)}\longrightarrow 2,
\]
hence the exponent $1/2$ in the first inequality is also optimal.

The construction and its proof were developed with the assistance of Hyra, an AI research agent based on the open-weights Hy3 model.
\end{abstract}

\maketitle

\section{Introduction}\label{sec:introduction}

Let $A$ be a finite nonempty subset of an abelian group.  Its doubling and
difference constants are
\[
 \sigma(A)=\frac{\card{A+A}}{\card A},
 \qquad
 \delta(A)=\frac{\card{A-A}}{\card A}.
\]

The classical sum--difference inequalities state
\begin{equation}\label{eq:classical-sum-difference}
 \sigma(A)^{1/2}\le \delta(A)\le \sigma(A)^2.
\end{equation}

For a finite set $A\subset\Z$ with $\card A\ge2$, define
\[
 C(A):=\frac{\log\sigma(A)}{\log\delta(A)}.
\]
The problem is to determine how large $C(A)$ can be, or equivalently to determine the least \textit{universal} exponent $c$ for which $\sigma(A)\leq \delta(A)^c$. 
Note that \eqref{eq:classical-sum-difference} immediately implies that, for every finite set \(A\subset\Z\) with \(\card A\ge2\),
\begin{equation}\label{eq:upper-bound}
C(A)\leq 2.
\end{equation}

Georgiev, G\'omez-Serrano, Tao, and Wagner studied this question using
AlphaEvolve~\cite{GeorgievEtAl2025}. In their study, they used the
Freiman--Pigarev value
\[
 \frac{\log(59/17)}{\log(55/17)}
 =1.059793\ldots
\]
as a baseline, and their AI-assisted search found a construction with
$C(A)\approx1.1219$. Under the present normalization, however, Penman and
Wells had previously constructed a sequence $(Q_j)$ of finite integer sets
for which
\[
 C(Q_j)\geq 1.125944
\]
for all sufficiently large $j$; see
\cite{penmanWells2013restricted,cutlerPebodySarkar2024}.

We prove that the classical upper bound \eqref{eq:upper-bound} is optimal.  
More precisely, our main result is the following.

\begin{theorem}\label{thm:main}
For every positive even integer $K$, there is an explicitly defined finite set
$A_K\subset\Z$ such that
\[
 C(A_K)>\frac{2K}{K+3}.
\]
Consequently, the least universal exponent is $2$.
\end{theorem}

The sets $A_K$ are described in \eqref{eq:R-explicit} and \eqref{eq:A-definition}. 
The proof combines a base-$12$ digit gadget, a three-state carry automaton, a symmetric additive basis in a cyclic group, and a Chinese remainder construction. 
We include all details below. We also provide a lean-based formal proof at \url{https://github.com/linhaowei1/sum-diff-proof}.

\section{Proof}\label{sec:proof}

\subsection{A base-12 digit gadget}

Let
\[
 W=\{0,1,2,4,5,9\}\subset\Z.
\]
Reducing the corresponding sumset and difference set modulo $12$, a direct
calculation gives
\begin{equation}\label{eq:seed}
 (W+W)\bmod 12=\Z/12\Z,
 \qquad
 (W-W)\bmod 12=(\Z/12\Z)\setminus\{6\}.
\end{equation}
As an integer difference set,
\[
 \Delta:=W-W
 =\{-9,-8,-7,-5,-4,-3,-2,-1,0,1,2,3,4,5,7,8,9\}.
\]

Set $n_0=1, Y_0=\{0\}$. For an integer $j\ge1$, put
\[
 n_j=12^j,
 \qquad
 Y_j=
 \left\{
   \sum_{i=0}^{j-1}w_i12^i:w_i\in W
 \right\}
 \subseteq\{0,1,\ldots,n_j-1\}.
\]
Then
\begin{equation}\label{eq:Ysize}
 \card{Y_j}=6^j.
\end{equation}

\begin{lemma}\label{lem:full-sum}
For every $j\ge1$,
\[
 (Y_j+Y_j)\bmod n_j=\Z/n_j\Z.
\]
\end{lemma}

\begin{proof}
Since \(Y_j=W+12Y_{j-1}\), the result follows by induction from
\[
(W+W)\bmod 12=\mathbb Z/12\mathbb Z,
\]
lifting one base-\(12\) digit at a time.
\end{proof}

We next estimate the modular difference set. For \(j\ge0\), define
\[
 t_j:=\card{(Y_j-Y_j)\bmod n_j}.
\]

\begin{lemma}\label{lem:difference-recurrence}
The sequence $t_j$ satisfies
\[
 t_0=1,\qquad t_1=11,
 \qquad
 t_{j+2}=13t_{j+1}-16t_j\quad(j\ge0).
\]
In particular, if
\[
 \lambda=\frac{13+\sqrt{105}}2,
 \qquad
 \mu=\frac{13-\sqrt{105}}2,
\]
then
\begin{equation}\label{eq:t-closed}
 t_j=
 \frac{\sqrt{105}+9}{2\sqrt{105}}\lambda^j
 +
 \frac{\sqrt{105}-9}{2\sqrt{105}}\mu^j
 <\lambda^j
 \qquad(j\ge1).
\end{equation}
\end{lemma}

\begin{proof}
Every residue modulo $12^j$ has a unique balanced expansion
\[
 \sum_{i=0}^{j-1}\varepsilon_i12^i,
 \qquad
 \varepsilon_i\in E:=\{-6,-5,\ldots,5\}.
\]
The residue represented by such a balanced digit string belongs to
\((Y_j-Y_j)\bmod 12^j\) if and only if there are carries
\[
 c_0=0,
 \qquad
 c_i\in\{-1,0,1\},
\]
such that
\begin{equation}\label{eq:carry-equation}
 \delta_i=\varepsilon_i+c_i-12c_{i+1}\in\Delta
 \qquad(0\le i<j).
\end{equation}
Indeed, summing \eqref{eq:carry-equation} after multiplication by $12^i$
shows that the two digit strings differ by $-c_j12^j$.

After any balanced prefix, the only possible nonempty sets of current carries
are
\[
 X_1=\{0\},
 \qquad
 X_2=\{-1,0\},
 \qquad
 X_3=\{0,1\}.
\]
For a current carry-set $X$ and a balanced digit $\varepsilon$, define
\[
 \Phi_\varepsilon(X)=
 \left\{
 c'\in\{-1,0,1\}:
 \varepsilon+c-12c'\in\Delta
 \text{ for some }c\in X
 \right\}.
\]
A check of the twelve possible values of $\varepsilon$ gives the transition
counts
\begin{equation}\label{eq:transition-table}
\begin{array}{c|ccc}
 &X_1&X_2&X_3\\ \hline
 X_1&5&3&3\\
 X_2&4&5&3\\
 X_3&4&4&4
\end{array}
\end{equation}
where rows are current states and columns are next states.  Thus, with
\[
 M=
 \begin{pmatrix}
 5&3&3\\
 4&5&3\\
 4&4&4
 \end{pmatrix},
 \qquad
 e_1=\begin{pmatrix}1\\0\\0\end{pmatrix},
 \qquad
 \one=\begin{pmatrix}1\\1\\1\end{pmatrix},
\]
we have
\[
 t_j=e_1^{\mathsf T}M^j\one.
\]
A direct multiplication gives
\[
 M^2-13M+16I_3
 =
 \begin{pmatrix}
 0&3&-3\\
 0&0&0\\
 0&-4&4
 \end{pmatrix},
 \qquad
 (M^2-13M+16I_3)\one=0.
\]
Therefore
\[
 t_{j+2}-13t_{j+1}+16t_j=0.
\]
The initial values are $t_0=1$ and $t_1=11$.  Solving the recurrence gives
\eqref{eq:t-closed}.  Both coefficients in \eqref{eq:t-closed} are positive
and sum to $1$, while $0<\mu<\lambda$, so $t_j<\lambda^j$ for $j\ge1$.
\end{proof}

The following elementary numerical estimate will be useful.

\begin{lemma}\label{lem:lambda-bound}
We have
\[
 \frac{\lambda}{12}<\frac{31}{32},
 \qquad
 \left(\frac{31}{32}\right)^{22}<\frac12.
\]
\end{lemma}

\begin{proof}
Straightforward calculation.
\end{proof}

\subsection{A symmetric additive basis in a cyclic group}

Fix a positive even integer $K$ and set
\begin{equation}\label{eq:parameters1}
 s=2^K+1,
 \qquad
 Q=s^2,
 \qquad
 m=\frac{s-1}{2}.
\end{equation}
Inside $\Z/Q\Z$, define
\[
 H=\{-m,-m+1,\ldots,m\},
 \qquad
 V=\{0,s,2s,\ldots,(s-1)s\},
\]
and put
\[
 I=H\cup V,
 \qquad
 B=I\setminus\{0\}.
\]
Since $H\cap V=\{0\}$,
\begin{equation}\label{eq:I-size}
 \rho:=\card I=2s-1,
 \qquad
 \card B=2s-2.
\end{equation}
Both $H$ and $V$ are symmetric, so $I=-I$ in $\mathbb Z/Q\mathbb Z$.

\begin{lemma}\label{lem:outer-basis}
The set $I$ is an additive basis of $\Z/Q\Z$:
\[
 I+I=\Z/Q\Z.
\]
Moreover, every $x\notin I$ lies in both $B+B$ and $B-B$.
\end{lemma}

\begin{proof}
Given $x\in\Z/Q\Z$, let $h\in H$ be the unique balanced representative of
$x$ modulo $s$.  Then $v=x-h$ is divisible by $s$, so $v\in V$, and
$x=h+v$.  Hence $I+I=\Z/Q\Z$.

If $x\notin I$, then neither $h$ nor $v$ is zero.  Thus $h,v\in B$ and
$x=h+v\in B+B$.  Since $-v\in B$, we also have $x=h-(-v)\in B-B$.
\end{proof}

\subsection{The CRT construction}

For a fixed positive even integer $K$, set
\begin{equation}\label{eq:parameters2}
 d=22(K+2),
 \qquad
 n=12^d,
 \qquad
 q=Qn,
 \qquad
 Y=Y_d,
 \qquad
 t=t_d.
\end{equation}
Because $K$ is even, $2^K\equiv1\pmod3$; $s=2^K+1\equiv2\pmod3$. 
Also $s$ is odd.  Therefore $\gcd(Q,n)=1$, and the Chinese remainder theorem
gives
\[
 \Z/q\Z\cong \Z/Q\Z\times\Z/n\Z.
\]

In this product group define
\begin{equation}\label{eq:Rcal}
 \mathcal R=
 \bigl(\{0\}\times\Z/n\Z\bigr)
 \cup
 \bigl(B\times Y\bigr).
\end{equation}
Let $R\subseteq\{0,1,\ldots,q-1\}$ be the set of least nonnegative
representatives of the inverse image of $\mathcal R$.  Equivalently,
\begin{equation}\label{eq:R-explicit}
\begin{split}
 R={}&\{x\in\Z:0\le x<q,\ x\equiv0\pmod Q\}\\
 &\cup
 \{x\in\Z:0\le x<q,\ x\bmod Q\in B,\ x\bmod n\in Y\}.
\end{split}
\end{equation}
Finally define the finite integer set $A_K$ (abbreviated to $A$ when $K$ is
fixed) by
\begin{equation}\label{eq:A-definition}
 \boxed{A_K=R\cup(R+q)=R+q\{0,1\}.}
\end{equation}
The two copies are disjoint, and hence
\begin{equation}\label{eq:A-size}
 \card A=2\card R
 =2\bigl(n+(\rho-1)6^d\bigr).
\end{equation}

\subsection{Exact modular counts}

\begin{lemma}\label{lem:mod-sum}
Modulo $q$, the sumset of $R$ is the whole group:
\[
 \card{(R+R)\bmod q}=q=Qn.
\]
\end{lemma}

\begin{proof}
Work in $\Z/Q\Z\times\Z/n\Z$.  For outer coordinate $0$, the sum of the
full zero row with itself is full.  For an outer coordinate $b\in B$, the sum
of the full zero row and the row $\{b\}\times Y$ is full in the inner
coordinate.  Finally, if $x\notin I$, Lemma~\ref{lem:outer-basis} gives
$x=b_1+b_2$ with $b_1,b_2\in B$, and Lemma~\ref{lem:full-sum} gives
$(Y+Y)\bmod n=\Z/n\Z$. Thus every outer coordinate has a full inner fiber.
\end{proof}

\begin{lemma}\label{lem:mod-diff}
Modulo \(q\), the difference set of \(R\) has cardinality
\[
 \card{(R-R)\bmod q}=\rho n+(Q-\rho)t.
\]
\end{lemma}

\begin{proof}
Again work in the CRT product.  The outer coordinates in $I$ have full inner
fibers.  For outer coordinate $0$, use the full zero row minus itself.  For
$b\in B$, use the row $\{b\}\times Y$ minus the full zero row; symmetry gives
the same conclusion for all of $I=-I$.

Now let $x\notin I$.  Any difference with outer coordinate $x$ must use two
nonzero rows, because using the zero row would give an outer coordinate in
$I$.  Hence its inner fiber is contained in $Y-Y$.  Conversely,
Lemma~\ref{lem:outer-basis} gives $x=b_1-b_2$ with $b_1,b_2\in B$, so every
element of $Y-Y$ occurs.  Thus each of the $Q-\rho$ outer coordinates outside
$I$ has an inner fiber of size exactly $t$.
\end{proof}

\subsection{Lifting the modular counts to integers}

\begin{lemma}\label{lem:lifting}
Let $R\subseteq\{0,1,\ldots,q-1\}$ and let $A=R+q\{0,1\}$.  Then
\[
 \card{A+A}\ge3\card{(R+R)\bmod q},
 \qquad
 \card{A-A}\le4\card{(R-R)\bmod q}.
\]
\end{lemma}

\begin{proof}
Fix a residue $c\pmod q$ occurring in $R+R$.  Its ordinary representatives in
$R+R$ have the form $c+kq$ with $k$ in a nonempty subset of $\{0,1\}$.  The
block sums from $q\{0,1\}+q\{0,1\}$ have indices $\{0,1,2\}$.  Therefore the
integers in $A+A$ above the residue $c$ have at least three distinct quotient
indices.  Summing over residues proves the first inequality.

For differences, the ordinary representatives in $R-R$ above a fixed residue
have quotient indices in a nonempty subset of $\{-1,0\}$, while the block
differences have indices $\{-1,0,1\}$.  The sum of these two index sets has at
most four elements.  Summing over the modular difference residues proves the
second inequality.
\end{proof}

\subsection{The final estimate}

Define the two small parameters
\[
 \alpha=\frac{(\rho-1)6^d}{n},
 \qquad
 \beta=\frac{t}{n}.
\]
Using \eqref{eq:I-size}, \eqref{eq:parameters1}, and
\eqref{eq:parameters2}, we obtain
\begin{equation}\label{eq:alpha}
 \alpha=(2s-2)2^{-d}=2^{K+1-d}=2^{-21K-43}<\frac12.
\end{equation}
By Lemmas~\ref{lem:difference-recurrence} and~\ref{lem:lambda-bound},
\begin{equation}\label{eq:beta}
 \beta
 <\left(\frac{\lambda}{12}\right)^d
 <\left(\frac{31}{32}\right)^{22(K+2)}
 <2^{-(K+2)}.
\end{equation}

\begin{proof}[Proof of Theorem~\ref{thm:main}]
By Lemma~\ref{lem:mod-sum}, Lemma~\ref{lem:lifting}, and
\eqref{eq:A-size},
\[
 \sigma(A)
 =\frac{\card{A+A}}{\card A}
 \ge\frac{3Qn}{2n(1+\alpha)}
 =\frac{3Q}{2(1+\alpha)}
 >Q=s^2,
\]
where the last inequality uses $\alpha<1/2$.

Likewise, Lemmas~\ref{lem:mod-diff} and~\ref{lem:lifting} give
\begin{align*}
 \delta(A)
 &\le
 \frac{4\bigl(\rho n+(Q-\rho)t\bigr)}{2n(1+\alpha)}\\
 &=\frac{2\bigl(\rho+(Q-\rho)\beta\bigr)}{1+\alpha}
 <2\rho+2Q\beta.
\end{align*}
Since $s=2^K+1<2^{K+1}$, \eqref{eq:beta} implies
\[
 Q\beta
 <\frac{s^2}{2^{K+2}}
 <\frac{s}{2}.
\]
As $\rho=2s-1$, it follows that
\[
 \delta(A)<2(2s-1)+s=5s-2<5s<8s.
\]
Therefore
\[
 C(A)
 =\frac{\log\sigma(A)}{\log\delta(A)}
 >\frac{2\log s}{\log(8s)}
 =\frac{2\log s}{\log s+3\log2}.
\]
The function
\[
 x\longmapsto\frac{2x}{x+3\log2}
\]
is strictly increasing for $x>0$, and $\log s>K\log2$.  Hence
\[
 C(A)>\frac{2K\log2}{K\log2+3\log2}
 =\frac{2K}{K+3}.
\]
Letting the positive even integer $K$ tend to infinity shows that no universal
exponent smaller than $2$ can work, while the classical upper bound in
\eqref{eq:classical-sum-difference} shows that the exponent $2$ does work.
\end{proof}

\begin{corollary}\label{cor:supremum}
Among finite subsets $A\subset\Z$ with $\card A\ge2$,
\[
 \sup_A C(A)=2.
\]
\end{corollary}

\begin{proof}
The equality characterization of Staps~\cite{Staps2015} gives $C(A)<2$ for
every admissible $A$.  On the other hand, Theorem~\ref{thm:main} gives a family
with
\[
 C(A_K)>\frac{2K}{K+3}\longrightarrow2
\]
as the positive even integer $K$ tends to infinity.
\end{proof}

\section{Use of AI Tools}\label{sec:ai-tools}

\begin{table}[!ht]
\centering
\caption{Published and manuscript-internal result comparison.}
\vspace{-0.5em}
\label{tab:search-milestones}
\scriptsize
\setlength{\tabcolsep}{3pt}
\renewcommand{\arraystretch}{1.0}
\begin{tabularx}{0.8\textwidth}{@{}
  p{1.25cm}
  >{\raggedright\arraybackslash}X
  r@{}}
\toprule
Date & Method, reference, and system & Value / bound \\
\midrule
\multicolumn{3}{@{}l}{\emph{Published mathematical constructions}}\\
\addlinespace[2pt]
Unknown
& Conway eight-point MSTD set
  \cite{nathanson2017mstd,hegarty2007explicit}
& $1.0344$ \\

1969
& Marica~\cite{marica1969conway}
& $1.0290$ \\

1973
& Freiman--Pigarev~\cite{freiman1973invariants}
& $1.0598$ \\

2013
& Penman--Wells family
  \cite{penmanWells2013restricted,cutlerPebodySarkar2024}
& $1.1259$ \\

\addlinespace[3pt]
\multicolumn{3}{@{}l}{\emph{Published agent-search results}}\\
\addlinespace[2pt]

2025-11
& AlphaEvolve (Gemini 2.0 Pro + Flash)
  \cite{GeorgievEtAl2025}
& $1.1219$ \\

2025-12
& LoongFlow (DeepSeek-R1-250528)
  \cite{wanEtAl2025loongflow,loongflowRepository}
& $1.1035$ \\

2026-04
& SimpleTES (gpt-oss-120b)~\cite{simpletes2026}
& $1.1440$ \\

2026-04
& SimpleTES with post-training (gpt-oss-120b)
  \cite{simpletes2026}
& $1.1449$ \\

2026-07
& OpenHands (GPT-5.4 medium)
  \cite{zhu2026evomaster}
& $1.0786$ \\

2026-07
& OpenClaw (GPT-5.4 medium)
  \cite{zhu2026evomaster}
& $1.1099$ \\

2026-07
& Codex (GPT-5.4 medium)
  \cite{zhu2026evomaster}
& $1.1121$ \\

2026-07
& EvoMaster (GPT-5.4 medium)~\cite{zhu2026evomaster}
& $1.1207$ \\

\addlinespace[3pt]
\multicolumn{3}{@{}l}{\emph{Exploratory runs reported only in this manuscript}}\\
\addlinespace[2pt]
2026-07
& Claude Fable 5
& $1.1133$ \\

2026-07
& Codex (GPT-5.5) with human guidance
& $1.2851$ \\

\addlinespace[3pt]
\multicolumn{3}{@{}l}{\emph{Theorem proved in this work}}\\
\addlinespace[2pt]
2026-07
& Explicit family $A_K$ (this work; Hyra with Hy3)
& $\sup C(A)=2$ \\
\bottomrule
\end{tabularx}
\end{table}

AI tools supported the exploratory and optimization stages of this work; the mathematical claims rest on the explicit construction and self-contained proofs above. 
Following the SimpleTES framework~\cite{simpletes2026}, we used
Hyra~\cite{hyra2026}, powered by the Hy3 model, to optimize finite-set constructions from scratch, raising the best value in its autonomous finite-search trajectory from approximately $1.14$ to $1.21$.

SimpleTES evaluates $C(A)$ from an explicitly enumerated Python list, so memory and enumeration costs limit the size of searchable sets. 
We therefore allowed agents to propose constructions and supporting arguments in natural language.
Using GPT-5.6 Sol as the judge, we ran Hyra for approximately 24 hours, producing the construction underlying the paper.
The LLM-based judgment was used only to guide exploration, not to certify mathematical correctness. 
We then independently checked the construction and proof, corrected the exposition, and prepared the argument presented here manually.

For context, we also tested Claude Fable 5 and Codex (GPT-5.5) with human guidance. Table \ref{tab:search-milestones} places the manuscript-internal runs alongside published constructions and agent-search results. We also used GPT-5.6 Sol to convert natural-language proofs into verifiable Lean 4 formal proofs (\url{https://github.com/linhaowei1/sum-diff-proof}).

\FloatBarrier

\section{Conclusion}\label{sec:conclusion}

We determine the optimal universal exponent relating sumsets and difference
sets. For every positive even integer $K$, our explicit construction satisfies
\[
C(A_K)>\frac{2K}{K+3},
\]
and hence $C(A_K)\to 2$ as $K\to\infty$. Together with the classical inequality $\sigma(A)^{1/2}\leq \delta(A)$, this proves that the least exponent $c$ for which $\sigma(A)\leq \delta(A)^c$ holds for every finite nonempty set $A\subset\mathbb Z$ is exactly $2$.

The supremum is approached but, by Staps' characterization of the equality cases, is not attained by any integer set with at least two elements \cite{Staps2015}. 

\clearpage
\section*{Acknowledgments}

We thank Letian Huang and Shuo Zhou from Peking University, Baihe Huang from UC Berkeley, and Honghao Lin from Carnegie Mellon University for helpful discussions. We also thank the Hyra team at Tencent Hunyuan for their work on Hyra.

\FloatBarrier

\bibliographystyle{amsplain}
\bibliography{reference}

\end{document}